\documentclass[11pt]{article}

\title{Finding a Hamilton cycle fast on average using rotations and extensions}

\author{
Yahav Alon
\thanks{‡School of Mathematical Sciences, Raymond and Beverly Sackler Faculty of Exact Sciences, Tel Aviv University,
Tel Aviv, 6997801, Israel. Email: yahavalo@mail.tau.ac.il.}
\and Michael Krivelevich
\thanks{School of Mathematical Sciences, Raymond and Beverly
Sackler Faculty of Exact Sciences, Tel Aviv University, Tel Aviv,
6997801, Israel. Email: krivelev@tauex.tau.ac.il. Partially supported by USA-Israel BSF grant 2014361, and by ISF grant 1261/17.}
}

\usepackage{tikz}
\usetikzlibrary{arrows,%
                petri,%
                topaths}%
\usepackage{tkz-berge}
\usepackage[position=top]{subfig}
\usepackage{amsmath,amsthm, amssymb,latexsym, amsfonts}
\usepackage{algorithm}
\usepackage[noend]{algpseudocode}

\begin{document}
\maketitle
\newtheorem{thm}{Theorem}
\newtheorem{propos}{Proposition}
\newtheorem{defin}{Definition}
\newtheorem{lemma}{Lemma}[section]
\newtheorem{corol}{Corollary}
\newtheorem{thmtool}{Theorem}[section]
\newtheorem{corollary}[thmtool]{Corollary}
\newtheorem{lem}[thmtool]{Lemma}
\newtheorem{defi}[thmtool]{Definition}
\newtheorem{prop}[thmtool]{Proposition}
\newtheorem{clm}[thmtool]{Claim}
\newtheorem{conjecture}{Conjecture}
\newtheorem{problem}{Problem}
\newcommand{\Proof}{\noindent{\bf Proof.}\ \ }
\newcommand{\Remarks}{\noindent{\bf Remarks:}\ \ }
\newcommand{\Remark}{\noindent{\bf Remark:}\ \ }
\newcommand{\cupdot}{\mathbin{\mathaccent\cdot\cup}}

\begin{abstract}
We present an algorithm \emph{CRE}, which either finds a Hamilton cycle in a graph $G$ or determines that there is no such cycle in the graph. The algorithm's expected running time over input distribution $G\sim G(n,p)$ is $(1+o(1))n/p$, the optimal possible expected time, for $p=p(n) \geq 70n^{-\frac{1}{2}}$. This improves upon previous results on this problem due to Gurevich and Shelah, and to Thomason.
\end{abstract}

\section{Introduction} \label{sec-intro} 

Hamilton cycles are a central topic in modern graph theory, a fact that extends to the field of random graphs as well, with numerous and diverse results regarding the appearance of Hamilton cycles in random graphs obtained over many years.\\
Consider the random graph model $G(n,p)$, in which every one of the edges of $K_n$ is added to $G$ with probability $p$ independently of the other edges. A classical result by Koml{\'o}s and Szemer{\'e}di \cite{KS83}, and independently by Bollob{\'a}s \cite{B84}, states that a random graph $G\sim G(n,p)$, with $np-\ln n - \ln \ln n \rightarrow \infty$, is with high probability  Hamiltonian. It should also be noted that if $np -\ln n - \ln \ln n \rightarrow -\infty$ then with high probability $\delta (G) \leq 1$, and thus $G$ is not Hamiltonian.\\
In fact, a stronger result was proved by Bollob{\'a}s in \cite{B84} and by Ajtai, Koml{\'o}s and Szemer{\'e}di in \cite{AKS85}. It states that the \textit{hitting time} of graph Hamiltonicity is with high probability equal to the hitting time of the property $\delta (G) \geq 2$. In other words: if one adds edges to an empty graph on $n$ vertices in a random order, then with high probability the exact edge whose addition to the graph has increased its minimal degree to $2$, has also made the graph Hamiltonian.\\
In light of this, one can ask whether there exists a computationally efficient way to find a Hamilton cycle in a graph $G$, or to determine that it contains none, provided that $G$ is sampled from the probability space $G(n,p)$ with $np-\ln n - \ln \ln n \rightarrow \infty$.\\
The answer to this question differs greatly depending on how one defines the term ``computationally efficient".\\
For example, if our interest lies in finding an algorithm with a fast worst case time complexity, that is, its running time on any input is bounded by some ``small" function of the number of vertices $n$, we might get disappointed. This is due to the fact that the graph Hamiltonicity problem is a well known NP-complete problem (see e.g. \cite{GJL}), and as such no polynomial time algorithm solving it is known. In fact, the best known worst case complexity algorithm is achieved by dynamic programming algorithms (see Bellman \cite{BELL} and Held, Karp \cite{HK}), with asymptotic time $O\left( 2^n\cdot n^2 \right)$.\\
That said, different models of complexity may yield very different results. Consider for example a model in which an algorithm is allowed to return the result \emph{``failure"}, admitting that it has failed to find a Hamilton cycle in the input graph (without providing a proof that there is none), under the condition that if $p \geq f(n)$ and $G\sim G(n,p)$ then the probability that the algorithm fails on input $G$ is of order $o(1)$.\\
In this model, much faster algorithms are available. A notable example is given in a 1987 paper by Bollob{\'a}s, Fenner and Frieze \cite{BFF}, who present an algorithm \emph{HAM1} with time complexity $O\left( n^{4+\varepsilon} \right)$ with $\varepsilon > 0$ arbitrarily small, that either finds a Hamilton cycle or returns \emph{``failure"}. They further show that if the input graph $G$ is distributed $G\sim G(n,p)$, for any $p=p(n)$, then
$$
\lim _{n\rightarrow \infty} Pr[\mbox{\emph{HAM1} finds a Hamilton cycle in }G] = \lim _{n\rightarrow \infty} Pr[G\ \mbox{is Hamiltonian}].
$$
Combined with the above stated fact that if $np-\ln n - \ln \ln n \rightarrow \infty$ then $G$ is with high probability Hamiltonian, this means that for $p \geq \frac{\ln n + \ln \ln n + \omega (1)}{n}$ the probability that \emph{HAM1} returns \emph{``failure"} is indeed $o(1)$.\\
Another example of a fast algorithm that is not likely to return \emph{``failure"} is given in \cite{FKSV}, where the authors choose to measure the complexity by the number of positive edge query results the algorithm requires. They show an algorithm that requires $(1+o(1))n$ successful queries, and fails with probability $o(1)$ on graphs distributed according to $G(n,p)$, with $p \geq \frac{\ln n + \ln \ln n + \omega (1)}{n}$.\\

An intuitive measure of complexity which seems interesting to consider is the \emph{expected} running time. Denote by $T_A(G)$ the running time of some algorithm \emph{A} on an input graph $G$. Say $G\sim G(n,p)$, how small can $\mathbb{E}\left[ T_A(G) \right]$ be?\\
If we assume that there is no polynomial time algorithm that finds a Hamilton cycle in a graph, then finding an algorithm with polynomial expected running time is in some sense a more difficult problem than that of finding a polynomial time algorithm that fails with probability $o(1)$: if the expected time is polynomial, it means that those cases on which the running time is super-polynomial take up at most $n^{-\omega (1)}$ of the probability space. So such an algorithm can be used to construct a polynomial time algorithm that returns \emph{``failure"} with probability $n^{-\omega (1)}$.\\
Bollob{\'a}s, Fenner and Frieze \cite{BFF} used their algorithm \emph{HAM1} to construct a slightly modified algorithm \emph{HAM}, which applies an exponential running time algorithm on inputs on which \emph{HAM1} returned \emph{``failure"}, and prove that the expected running time of \emph{HAM} on $G\sim G\left( n,\frac{1}{2} \right)$ is polynomial in $n$.\\
Gurevich and Shelah \cite{GS87} improved upon this result, by presenting an algorithm \emph{HPA}, which finds a Hamiltonian $s-t$ path in a graph $G$, with a linear expected running time, where this time the input is assumed to be distributed according to distribution $G(n,p)$, with $p\in [0,1]$ being a constant (not necessarily $\frac{1}{2}$). This can easily be altered into an algorithm that finds a Hamilton cycle rather than a Hamilton $s-t$ path. They did this by presenting three consecutive algorithms \emph{HPA1, HPA2, HPA3}, such that failure of one algorithm to find a Hamilton $s-t$ path results in the next one being called, and such that \emph{HPA1} takes linear time and $$Pr[\mbox{\emph{HPAi} fails on }G]\cdot \mathbb{E}\left[ T_{HPA(i+1)}(G)\right] = O(n).$$
They further show that their result is optimal for this range of $p$, by proving a stronger claim: If $A$ is an algorithm for finding a Hamilton cycle and $p \geq \frac{3 \ln n}{n}$, $G\sim G(n,p)$, then $\mathbb{E}\left[ T_A(G)\right] \geq n / p$. This result can be obtained by observing that in order to find a Hamilton cycle in a graph $G$, the algorithm must sample at least $n$ existing edges of $G$, which means that the expected number of queried pairs of vertices in $A$ must be at least the expected number of queries required for finding $n$ edges, which is exactly $n/p$.\\
Further improvement was later given by Thomason \cite{THOM}, who presented an algorithm \emph{A}, similarly constructed of three consecutive algorithms \emph{A1,A2,A3}. The expected running time of \emph{A} is asymptotically optimal up to multiplication by a constant (that is $\mathbb{E}\left[ T_A(G) \right] = O(n/p)$), for a wider class of random graphs: whenever $p\geq 12n^{-\frac{1}{3}}$.\\
For further reading on the algorithmic aspects of random graphs, including Hamiltonicity, we refer to \cite{FM}.\\
In this paper we present a new algorithm \emph{CRE} (Cycle rotation extension) for finding a Hamilton cycle, and prove that if $p\geq 70 n^{-\frac{1}{2}}$ and $G\sim G(n,p)$ then $\mathbb{E}\left[ T_{CRE}(G) \right] = (1+o(1))n/p$. This constitutes a substantial progress in a long-standing open problem on Hamiltonicity of random graphs (see e.g., \textbf{Problem 16} in \cite{FRI}).\\
\noindent Formally, we prove the following main result:

\begin{thm}
Let $p\geq 70 n^{-\frac{1}{2}}$ and let $G\sim G(n,p)$. There is an algorithm for finding a Hamilton cycle in a graph, with expected running time $(1+o(1))n/p$ on $G$.
\end{thm}

As the algorithm's name suggests, we will try and employ techniques inspired by P{\'o}sa's \emph{rotation-extension}, which were introduced by P{\'o}sa in 1976 \cite{POS} in his research of Hamiltonicity in random graphs. Informally put, \emph{rotation-extension} is a technique which under certain conditions allows one to gradually extend paths or cycles in a graph, by finding (through a process usually referred to as a rotation) a large number of pairs of vertices, such that the existence of an edge between any of these pairs enables one to get a longer path or cycle (an extension) using this edge.\\
Similarly to the previous results, we will define \emph{CRE} by aligning three algorithms, each calling the next one in case of failure. In essence, the three algorithms will be:
\begin{itemize}
	\item \emph{CRE1} -- A simple greedy algorithm, tasked with optimizing the expected time complexity.
	\item \emph{CRE2} -- The main algorithm, tasked with finding a Hamilton cycle in polynomial time in all but an exponentially small fraction of the probability space.
	\item \emph{CRE3} -- An exponential running time algorithm tasked with finding a Hamilton cycle in the graph when the previous two algorithms failed. This algorithm is identical to \emph{HPA3}.
\end{itemize}

\noindent In Section \ref{sec-per} we present some preliminaries. In Section \ref{sec-alg} we present the \emph{CRE} algorithm, and prove its correctness. In Section \ref{sec-time} we prove that the expected running time of \emph{CRE} is $(1+o(1))n/p$. In Section \ref{remarks} we add some concluding remarks.

\section{Preliminaries} \label{sec-per}

In this section we provide several definitions and results to be used in the following sections.\\
Throughout the paper, it is assumed that all logarithmic functions are in the natural base, unless explicitly stated otherwise.\\
We suppress the rounding notation occasionally to simplify the presentation.\\
The following standard graph theoretic notations will be used:
\begin{itemize}
\item $N_G(U)$ : the external neighbourhood of a vertex subset $U$ in the graph $G$, i.e.
$$
N_G(U) = \lbrace v \in V(G)\setminus U:\ v\ \mbox{has\ a\ neighbour\ in}\ U \rbrace.
$$
\item $e_G(U)$: the number of edges spanned by a vertex subset $U$ in a graph $G$. This will sometimes be abbreviated as $e(U)$, when the identity of $G$ is clear from the context.
\item $e_G(U,W)$: the number of edges of $G$ between the two disjoint vertex sets $U,W$. This will sometimes be abbreviated as $e(U,W)$ when $G$ is clear from the context.
\end{itemize}
\noindent Furthermore, given a cycle or a path $S$ in a graph, with some orientation, we denote:
\begin{itemize}
\item $S^{-1}$: the cycle composed of the vertices and edges of $S$, but with the opposite orientation.
\item $s_S(v)$: the successor of a vertex $v \in S$ on $S$, according to the given orientation. When the identity of the cycle is clear, we will write $s(v)$.
\item $s_S(U)$: the set of successors $\{s_S(u):\ u\in U\}$. When the identity of the cycle is clear, we will write $s(U)$.
\item $p_S(v)$: the predecessor of a vertex $v \in S$ on $S$, according to the given orientation. When the identity of the cycle is clear, we will write $p(v)$.
\item $p_S(U)$: the set of predecessors $\{p_S(u):\ u\in U\}$. When the identity of the cycle is clear, we will write $p(U)$.
\item $S(v \rightarrow u)$: the path $\left( v,s_S(v),s^2_S(v),...,p_S(u),u \right) \subseteq S$.
\end{itemize}

\noindent Gearing towards our concrete setting of a graph $G$ distributed according to $G(n,p)$ with $p\geq 70n^{-\frac{1}{2}}$, given a graph $G$, we will define the set of vertices with small degree (with regards to the expected degree) in $G$:
\begin{defin}\label{small}
Let $G$ be a graph on $n$ vertices. The set $\mathit{SMALL}(G)$ is defined as
$$\mathit{SMALL}(G):= \lbrace v\in V(G) \mid d(v) < 40\sqrt{n} \rbrace .$$
\end{defin}
\noindent We shall also make use of the following definition:\\

\begin{defin}
Let $\Gamma = \left( X \cup Y,E \right)$ be a bipartite graph. An edge subset $M \subseteq E(\Gamma )$ is called a \emph{$\leq 2$-matching} from $X$ to $Y$ if each vertex of $X$ is incident to at most $2$ edges in $M$, and each vertex of $Y$ is incident to at most one edge in $M$. A \emph{maximum $\leq 2$-matching} in $\Gamma$ is a $\leq 2$\emph{-matching} with the maximum possible number of edges.
\end{defin}

\noindent We note that given a bipartite graph $\Gamma = \left( X \cup Y ,E\right)$, a \emph{maximum $\leq 2$-matching} from $X$ to $Y$ can be found in time $(|X|+|Y|)^{O(1)}$ by using the \emph{MaxFlow} algorithm.\\

\noindent For some of our probabilistic bounds, we will use the following standard result throughout the paper:
\begin{lem}\label{chernoff}
{\em (Chernoff bound for binomial tails, see e.g. \cite{CHER})} Let $X\sim Bin(n,p)$. Then for every $\delta > 0$, $Pr[X < np - \delta ] \leq \exp \left( -\frac{\delta ^2}{2np} \right) .$
\end{lem}

\section{The \emph{CRE} algorithm} \label{sec-alg}
We now present the three components of the \emph{CRE} algorithm, and prove that they are sound. Recall that each component can either fail or return a result, which is either a Hamilton cycle in the input graph or a declaration that there is none. The \emph{CRE} algorithm itself will be:\\
$CRE(G)$:\\
\indent \emph{If} $CRE1(G)$ \emph{did not fail, return the result of} $CRE1(G)$. \emph{Otherwise:}\\
\indent \emph{If} $CRE2(G)$ \emph{did not fail, return the result of} $CRE2(G)$. \emph{Otherwise:}\\
\indent \emph{Return the result of} $CRE3(G)$.

\subsection{\emph{CRE1}} \label{subalg1}
\noindent We present the algorithm \emph{CRE1}. This algorithm will be a greedy algorithm, tasked with optimizing the expected running time. As such, we aim for it to have the following properties, whenever $p\geq 70 n^{-\frac{1}{2}}$:
\begin{itemize}
	\item $\mathbb{E}\left[ T_{CRE1}(G) \right] =(1+o(1))n / p$;
	\item $Pr[CRE1\mbox{ returns }``failure"] \cdot \mathbb{E}\left[ T_{CRE2}(G) \right] = o(n/p)$.
\end{itemize}
In the algorithm description we will assume that $V(G)=[n]$.\\

\noindent \textbf{The \emph{CRE1} algorithm description:}\\
\begin{itemize}

	\item[\textbf{Step 1.}] Attempt to construct a path $P_1$ in $G\left( [n/2] \right)$ by greedily querying for a neighbour of the current last vertex in the path from outside the path, until the path's end vertex does not have any neighbours among the remaining vertices. If $\frac{n}{2}-|P_1| > \sqrt{n} \log n$, return \emph{``Failure"}. Denote this path by $P_1 = (v_1,...,v_{n/2-n_1})$, with $n_1 = |[n/2]\setminus P_1|$.\\
	Attempt to construct a path $P_2$ in $G\left( [n/2+1,n] \right)$ in the same manner, and return \emph{``Failure"} if $\frac{n}{2}-|P_2| > \sqrt{n} \log n$. Denote $P_2 = (u_1,...,u_{n/2-n_2})$.
	
	\item[\textbf{Step 2.}] Find indices $i,j,k,l$ with minimal $i+j+k+l$, such that $(v_i,u_{n/2-n_2-j}),(v_{n/2-n_1-k},u_l) \in E(G)$. If $i+j+k+l > \sqrt{n} \log n$, return \emph{``Failure"}. Otherwise, denote by $S_0$ the cycle: $$S_0 := P_1(v_i\rightarrow v_{n/2-n_1-k})\cup \{ (v_{n/2-n_1-k},u_l)\} \cup P_2(u_l\rightarrow u_{n/2-n_2-j}) \cup \{(v_i,u_{n/2-n_2-j}) \}.$$
	
	\item[\textbf{Step 3.}] Initialize $i=0$, and repeat the following loop until no vertices are left outside the cycle $S_i$. Choose some vertex $v\notin S_i$. For ease of description we will assume that $v\in [n/2]$. In the complementing case, the description is completely symmetrical, replacing $P_2$ with $P_1$, $n_2$ with $n_1$ and so on.\\
	Create a set $X = \{x_1,...,x_{\sqrt[3]{n}}\}$ of neighbours of $v$ on $(P_2 \cap S_i)\setminus \{ u_{n/2-n_2-j} \}$ that have not been used in this step, with $z:=x_{\sqrt[3]{n}}$ being the maximal one with respect to $P_2$. Return \emph{``failure"} if no such $\sqrt[3]{n}$ vertices exist. Otherwise, create a set $Y=\{y_1,...,y_{\sqrt[3]{n}}\}$ of neighbours of $s_{S_i}( z )$ on $(P_1 \cap S_i) \setminus \{v_i\}$. Return \emph{``failure"} if no such $\sqrt[3]{n}$ vertices exist. Finally, find a pair $x\in X\setminus \{ z \}, y\in Y$ such that $\left( s_{S_i}(x), p_{S_i}(y) \right) \in E(G)$. If no such pair exists, return \emph{``failure"}. Otherwise, set
	\[
	\begin{array}{rl}
	S_{i+1}:= & \{(v,x)\} \cup S_i^{-1}(x\rightarrow y) \cup \{ (y,s(z)) \} \cup S_i(s(z) \rightarrow p(y)) \cup \{ (p(y),s(x)) \}\\
	& \cup S_i(s(x)\rightarrow z ) \cup \{ (z,v) \};
	\end{array}
	\]
	$i:=i+1$.

\end{itemize}

\subsection{\emph{CRE2}} \label{subalg2}

\noindent We present a description of \emph{CRE2}, followed by a proof that the algorithm is sound, that is, if \emph{CRE2} does not fail on a graph $G$ then it returns a Hamilton cycle that is a subgraph of $G$ if and only if $G$ is Hamiltonian.\\

\noindent \textbf{The \emph{CRE2} algorithm description:}\\
\begin{itemize}

\item[\textbf{Step 1.}] Determine $\mathit{SMALL}(G)$ (see Def. \ref{small}) by going over all vertices and checking their degrees in $G$. If the resulting set is larger than $ 2\sqrt{n}$, return \emph{``Failure"}.

\item[\textbf{Step 2.}] Find a \emph{maximum $\leq 2$-matching} $M$ in $G$ from $\mathit{SMALL}(G)$ to $V(G)\setminus \mathit{SMALL}(G)$. Denote by $U$ the subset of vertices in $V(G)\setminus \mathit{SMALL}(G)$ that have degree $1$ in $M$. If $|U| \leq |\mathit{SMALL}(G)|$, add arbitrary vertices to $U$ until it is of size $|\mathit{SMALL}(G)|+1$.

\item[\textbf{Step 3.}] Using the dynamic programming algorithm (\emph{HPA3}, see description in Section \ref{subalg3}), find a Hamilton cycle in the graph with vertex set $U\cup \mathit{SMALL}(G)$ and edge set $E_G\left( U\cup \mathit{SMALL}(G)\right) \cup \left( U \times U \right)$. If no such cycle exists, determine that $G$ is not Hamiltonian. Otherwise, denote this cycle by $C$.\\
Let $\mathit{NE} = (U\times U)\cap C \setminus E(G)$, let $|\mathit{NE}|=r$, and denote the members of $\mathit{NE}$ by $\{e_1,...,e_r\}$.

\item[\textbf{Step 4.}] For each $1 \leq j \leq r$ find a path $P_j$ of length at most $4$ connecting the two vertices of $e_j$, with all of its internal vertices in $G\setminus \left( \bigcup\limits _{k =1}^{j-1} P_k \cup \mathit{SMALL}(G) \cup U \right)$, using \emph{BFS}. If for some $j$ no such path exists, return \emph{``Failure"}. Otherwise, set $i=0$ and denote the resulting cycle by $S_0 = \left( C\cup \bigcup\limits _{j=1}^r P_j \right) \setminus \mathit{NE}$.

\item[\textbf{Step 5.}] Attempt to add at least one vertex of $V(G)\setminus V(S_i)$ to $S_i$ by doing the following:\\
Using BFS, determine all connected components of $G\setminus S_i$, and denote by $V_i$ a largest connected component. If $|S_i|\geq 0.99n$ and $|V_i| \leq 15 \sqrt{n}$, go to \emph{Step 6}. Otherwise, choose an arbitrary orientation to $S_i$ and let $U_i := s(N_G(V_i)\cap S_i)$. If $U_i$ is an independent set, return \emph{``Failure"}. Otherwise, let $(u,w)$ be an edge in $U_i$, let $u^{\prime} = p(u),w^{\prime} = p(w)$ and let $P$ be a path, with all its internal vertices in $V_i$, connecting $u^{\prime}$ to $w^{\prime}$ (this path was uncovered in the BFS stage). Without loss of generality, $u$ precedes $w$ on $S_i$. Set $S_{i+1}$ to be:\\
$$
S_{i+1}= S_i(w\rightarrow u^{\prime}) \cup P \cup S_i^{-1}( w^{\prime} \rightarrow u) \cup \lbrace (u,w) \rbrace .
$$
Set $i=i+1$, and return to \emph{Step 5}.

\item[\textbf{Step 6.}] While there is some vertex $v\in V(G)\setminus V(S_i)$, attempt to add it to $S_i$ by exhaustively searching for two vertices $u,w \in N_G(v)\cap S_i$, a set $E_1 \subseteq E(S_i)$ of size at most 4, and a set $E_2 \subseteq E_G(V(S_i))\setminus E(S_i)$ of size $|E_1|-1$, such that $S_{i+1} :=(S_i\setminus E_1) \cup E_2 \cup \{ (u,v) , (v,w) \}$ is a cycle of size $|S_i|+1$. If no such $u,w,E_1,E_2$ exist, return \emph{``failure"}.
\\

\end{itemize}

\begin{lemma}\label{sound}
If $G$ is a graph such that \emph{CRE2} does not result in failure when applied to $G$, then \emph{CRE2} returns a Hamilton cycle if and only if $G$ is Hamiltonian. Furthermore, if \emph{CRE2} returns a Hamilton cycle then it is a subgraph of $G$. 
\end{lemma}

\begin{proof}
In each step $E(S_i) \subseteq E(G)$ and $S_i\subsetneq S_{i+1}$. So it is clear that if the algorithm returns a Hamilton cycle then it is indeed a Hamilton cycle contained in $G$.\\
The complementing case is \emph{CRE2} declaring that $G$ is not Hamiltonian. This can only occur in Step 3, if the algorithm failed to find a Hamilton cycle in the graph consisting of vertices $\mathit{SMALL}(G)\cup U$ and edges $E_G(\mathit{SMALL}(G) \cup U)\cup (U\times U)$, which we will denote by $H$. Since the dynamic programming algorithm was used to find such a cycle, failure to find one means that it does not exist in $H$, so it remains to be shown that if $G$ is Hamiltonian then $H$ must also be Hamiltonian. We provide a proof of this due to Thomason \cite{THOM}.\\
Let $G^*$ denote the graph obtained by adding to $G$ all the non-edges with both vertices in $G\setminus \mathit{SMALL}(G)$. Assume that $G$ is Hamiltonian. Then $G^*$ must also be Hamiltonian.\\
For some Hamilton cycle $C$, define its \emph{kernel set} to be the edge subset $C\setminus E_{G^*}(V(G)\setminus \mathit{SMALL}(G))$. The kernel set of a Hamilton cycle consists of a set of disjoint paths in $G\setminus E_G(V(G)\setminus \mathit{SMALL}(G))$, containing between them all of $\mathit{SMALL}(G)$, whose endvertices lie in $V(G)\setminus \mathit{SMALL}(G)$.\\
Let $C$ be a Hamilton cycle in $G^*$ such that the number of edges from $M$ contained in its kernel set is maximised. Denote $\mathit{SMALL}(G) = W_0 \cupdot W_1 \cupdot W_2$, where $W_i$ is the subset of $\mathit{SMALL}(G)$ joined by $i$ edges of the kernel set to $V(G)\setminus \mathit{SMALL}(G)$. Let $K\subseteq V(G)\setminus \mathit{SMALL}(G)$ be the set of vertices joined by the kernel set to $\mathit{SMALL}(G)$. Then any vertex in $W_i$ matches to at most $2-i$ vertices in $U\setminus K$, for otherwise if $x\in W_i$ and $(x,y) \in M$, where $y\notin K$, we can remove a kernel set edge from $x$, replace it with $(x,y)$, and create a new kernel set (of another Hamilton cycle $C^{\prime}$) with more edges from $M$ in it. Now, for each vertex in $K$, choose an edge of the kernel set incident to it arbitrarily. Then a vertex of $W_i$ is incident with at most $i$ of these edges. So these edges, along with the edge set $M\cap (\mathit{SMALL}(G) \times (U\setminus K))$, together form a $\leq 2$-matching of order $|U\cup K|$. Since the largest $\leq 2$-matching has order exactly $|U|$, we see that $K\subseteq U$. It now follows from the definition of a kernel set that we can construct a Hamilton cycle in $H$, as claimed.\\
\end{proof}

\subsection{\emph{CRE3}} \label{subalg3}

The final part of \emph{CRE} is \emph{CRE3}, an algorithm with the following desired properties:
\begin{itemize}
	\item The time complexity of \emph{CRE3} is $2^{2n}\cdot n^{O(1)}$;
	\item The space complexity of \emph{CRE3} is linear in $n$;
	\item The result of \emph{CRE3} is either a Hamilton cycle contained in the input graph, or a declaration that the graph is not Hamiltonian if the input graph contains none.
\end{itemize}
Luckily, such an algorithm already exists --- the algorithm \emph{HPA3} presented by Gurevich and Shelah in \cite{GS87}. For completeness we give a brief description of the algorithm. For proof of the properties, see the original paper. We note that, as mentioned in Section \ref{sec-intro}, an algorithm with time complexity $O\left( 2^n\cdot n^2 \right)$ is known. The downside of this algorithm is that it also has exponential space complexity. This is not a very big issue for us, since our interests in this paper lie exclusively in time complexity, but since we can get a similar algorithm, but with linear space, with its time complexity still sufficiently small for our purposes, this is the one we chose.\\
\noindent The algorithm \emph{HPA3}, given a graph $G$ and two vertices $s,t\in V(G)$, finds a Hamilton path in $G$ from $s$ to $t$. First we note that converting this algorithm into an algorithm for finding a Hamilton cycle is very simple: choose an arbitrary vertex in $G$, say $s$, and iterate $HPA3(G\setminus (s,t),s,t)$ over all $t\in N_G(s)$. If for some $t$ a Hamilton $s-t$ path $P$ is found then $P\cup (s,t)$ is a Hamilton cycle in $G$. If all iterations fail, then surely $G$ cannot be Hamiltonian.\\
\noindent \emph{HPA3} is defined recursively, as follows:\\
$HPA3(G,s,t):$\\
\emph{If $V(G)=\{s,t\}$, return $(s,t)$ if it is an edge, and} \textbf{``No such path"} if it is not an edge. \emph{Otherwise:}\\
\emph{For all $c\in V(G)\setminus \{s,t\}$ and for all $A\subseteq V(G)\setminus \{s,t,c\}$ of size $\lfloor \frac{n-3}{2} \rfloor$:} \\
\indent \emph{If $HPA3(A,s,c)$ and $HPA3(G\setminus A,c,t)$ are successful, return $HPA3(A,s,c) \cup HPA3(G\setminus A,c,t)$};\\
\indent \emph{otherwise,  continue.}\\
\emph{If loop failed, return} \textbf{``No such path".}

\section{Expected time complexity of \emph{CRE}} \label{sec-time}

In this section we aim to prove that the algorithm described in Section \ref{sec-alg} meets the time complexity goals we had set, that is: if $p \geq 70n^{-\frac{1}{2}}$, then the expected running time over $G(n,p)$ is $(1+o(1))n/p$. Since
\[
\begin{array}{rcl}
\mathbb{E}\left[ T_{CRE}(G) \right] & \leq & \mathbb{E}\left[ T_{CRE1}(G) \right] + Pr[CRE1\mbox{ fails}]\cdot \mathbb{E}\left[ T_{CRE2}(G)\, |\, CRE1\mbox{ fails} \right] \\
 & & + Pr[CRE2\mbox{ fails}]\cdot \mathbb{E}\left[ T_{CRE3}(G) \right] ,
\end{array}
\]
it is sufficient to prove that the following hold:
\begin{itemize}
	\item $\mathbb{E}\left[ T_{CRE1}(G) \right] = (1+o(1))n/p$;
	\item $Pr[CRE1\mbox{ fails}]\cdot \mathbb{E}\left[ T_{CRE2}(G) \, |\, CRE1\mbox{ fails} \right] = o(n/p)$;
	\item The probability that \emph{CRE2} returns \emph{``failure"} is $2^{-2n} \cdot n^{-\omega (1)}$;
	\item The running time of \emph{CRE3} is $ 2^{2n} \cdot n^{O(1)}$.
\end{itemize}

\noindent A proof of the last point is provided in \cite{GS87}. We now provide proofs for the other three points.\\

\subsection{Expected running time of \emph{CRE1}} \label{subsec-time1}

\begin{lemma}\label{runtime1}
If $p\geq 70n^{-\frac{1}{2}}$, $G\sim G(n,p)$, then $\mathbb{E}\left[ T_{CRE1}(G) \right] = (1+o(1))n/p$.
\end{lemma}

\begin{proof}
The expected running time of \emph{CRE1} is the sum of the expected running times of its three steps.
\begin{itemize}

\item In Step one \emph{CRE1} samples edges, until it reaches at most $n-2$ successes, which means that the expected time of this step is at most $(n-2)/p$;

\item In Step 2 \emph{CRE1} samples edges until it finds two existing edges. So the expected running time of this step is $2/p$;

\item In Step 3 \emph{CRE1} repeats a loop at most $\sqrt{n}\log n$ times. In each time, it samples edges until it finds $2\sqrt[3]{n}+1$ existing ones. So the expected running time of this step is at most $\sqrt{n}\log n \cdot \left( 2\sqrt[3]{n} + 1 \right) /p = o(n/p)$.

\end{itemize}
Overall, we get the desired sum of $(1+o(1))n/p$.

\end{proof}

\subsection{Probability of failure of \emph{CRE1}} \label{subsec-prob1}

\begin{lemma}\label{prob1}
Let $p\geq 70n^{-\frac{1}{2}}$ and let $G\sim G(n,p)$. Then the probability that $CRE1(G)$ returns the result \emph{``failure"} is $o(n^{-60})$.
\end{lemma}

\begin{proof}
We note that since no edge is sampled twice during the run of \emph{CRE1}, all the possible events that lead to failure are independent. We bound from above the probability of each of these events occurring.
\begin{enumerate}
	\item \emph{CRE1} fails if at some point in Step 1 the last vertex in $P_1$ has no neighbours in the set $[n/2]\setminus P_1$, and if at that point this set is larger than $\sqrt{n}\log n$. The probability of this occurring is at most the probability that among $\frac{n}{2}$ independent random variables distributed $Bin\left( \sqrt{n}\log n,p \right)$ at least one is equal to zero. We bound this probability by applying the union bound:
	\[
		\begin{array}{rcl}
		Pr[n_1 \geq \sqrt{n}\log n] & \leq & 0.5n\cdot (1-p)^{\sqrt{n}\log n} \\
		& \leq & 0.5n\exp (-70\log n) \\
		& = & o(n^{-60}).
		\end{array}
	\]
	
	\item $$Pr[n_2 \geq \sqrt{n}\log n] = Pr[n_1 \geq \sqrt{n}\log n] = o(n^{-60}).$$
	
	\item Step 2 results in failure if the minimal indices $i,j,k,l$ for which $\left(v_i,u_{n/2-n_2-j} \right) ,\left(v_{n/2-n_1-k},u_l \right)$ are in $E(G)$ satisfy $i+j+k+l > \sqrt{n}\log n$, and in particular $i+j> 0.5 \sqrt{n}\log n$ or $k+l> 0.5 \sqrt{n}\log n$. There are $\binom{0.5 \sqrt{n}\log n}{2} \geq 0.1n\log ^2n$ pairs $i,j$ (or $k,l$) with $i+j \leq 0.5 \sqrt{n}\log n$, for which an edge query resulted in failure. Applying the union bound we get
	\[
		\begin{array}{rcl}
		Pr[i+j+k+l > \sqrt{n}\log n] & \leq & 2Pr[i+j > 0.5\sqrt{n}\log n] \\
		& \leq & (1-p)^{0.1n\log ^2 n} = n^{-\omega (1)}.
		\end{array}
	\]
	
	\item If Step 3 resulted in failure, say in the $m$'th iteration, then there was some vertex $v$ outside of $S_m$ such that one of the following happened:
	\begin{enumerate}
	\item $v$ did not have $\sqrt[3]{n}$ neighbours in (wlog) $(P_2 \cap S_m)\setminus \{ u_{n/2-n_2-j} \}$ that have not been used in iterations 0 to $i-1$;
	\item $s_{S_m}(z)$ did not have $\sqrt[3]{n}$ neighbours in $(P_1 \cap S_m)\setminus \{ v_i \}$;
	\item $s(X)$ and $p(Y)$ did not have any edge between them.
	\end{enumerate}
	
	Since up to the $m$'th iteration, at most $\sqrt{n}\log n \cdot \sqrt[3]{n} = o(n)$ vertices of $S_m$ have been used, the probability of \emph{(a)} and \emph{(b)} is at most the probability that $Bin(n/6,p)<\sqrt[3]{n}$. So:
	\[
		\begin{array}{rcl}
		Pr[\mbox{Step 3 failed}] & \leq & n\cdot \left( 2\cdot Pr\left[ Bin(n/6,p)<\sqrt[3]{n} \right] + (1-p)^{\sqrt[3]{n}(\sqrt[3]{n}-1)} \right) \\
		& \leq & n\cdot \left( \exp \left( -\Omega (\sqrt{n}) \right) + \exp \left( -\Omega (\sqrt[6]{n}) \right) \right) = n^{-\omega (1)}.
		\end{array}
	\]
	\end{enumerate}
	So all of the events that lead to failure have probability $o(n^{-60})$, and therefore the probability of failure is also $o(n^{-60})$, as we have set out to prove.

\end{proof}

\subsection{Expected running time of \emph{CRE2}} \label{subsec-time2}

\begin{lemma}\label{runtime2}
Let $p=p(n)\geq 70n^{-\frac{1}{2}}$. Then $Pr[CRE1\mbox{ fails}]\cdot \mathbb{E}\left[ T_{CRE2}(G)\, |\, CRE1\mbox{ fails} \right] = O(1)$, where the input to both algorithms is distributed according to $G\sim G(n,p)$.
\end{lemma}

\begin{proof}
Denote $Pr[CRE1\mbox{ fails}] := p_1$. Except for Step 3, all steps of \emph{CRE2} have time complexity at most $O(n^5)$, regardless of the input graph. As for Step 3, since $|U\cup \mathit{SMALL}(G)|\leq 3|\mathit{SMALL}(G)|$, the expected runtime of this step (assuming we reach it) is 
$$
\begin{array}{rcl}
\mathbb{E}\left[ T_{Step\ 3}(G)\, |\, CRE1\mbox{ fails} \right] & = &\sum _{k=1}^{2\sqrt{n}} k^{O(1)}2^{6k}\cdot Pr\left[ |\mathit{SMALL}(G)| = k\, | \, CRE1\mbox{ fails}\right]\\
& \leq & {p_1}^{-1} \cdot  \sum _{k=1}^{2\sqrt{n}} k^{O(1)}2^{6k}\cdot Pr\left[ |\mathit{SMALL}(G)| = k\right].
\end{array}
$$ We bound each term from above, using the Chernoff bound (Lemma \ref{chernoff})

\[
\begin{array}{rcl}
k^{O(1)}2^{6k}\cdot Pr\left[ |\mathit{SMALL}(G)| = k\right] & \leq & k^{O(1)}2^{6k} \cdot \binom{n}{k} \cdot Pr\left[ Bin(k(n-k),p) \leq \frac{3}{4} knp \right]  \\
& \leq & \exp \left( O(\log k) + 6k + k\log n - \Omega (knp) \right) = o\left( n^{-1} \right),
\end{array}
\]

\noindent hence the value of the entire sum above is at most $o(1)$.

\noindent So overall
$$
Pr[CRE1\mbox{ fails}]\cdot \mathbb{E}\left[ T_{CRE2}(G)\, |\, CRE1\mbox{ fails} \right] = p_1 \cdot O\left(  n^5 + {p_1}^{-1} \right) = O(1).
$$

\end{proof}

\subsection{Probability of failure of \emph{CRE2}} \label{subsec-prob2}

\noindent Let $G\sim G(n,p)$, where $p = p(n) \geq 70n^{-\frac{1}{2}}$.\\
We will call an event $A$ \emph{rare} if $Pr[A] = 2^{-2n} \cdot n^{-\omega (1)}$. Our goal is to prove that $CRE2(G)$ resulting in \emph{failure} is a rare event. We aim to do this by presenting a graph property $(P)$ such that:
\begin{itemize}
\item $G \notin (P)$ is rare;
\item If $G \in (P)$ then \emph{CRE2} deterministically either finds a Hamilton cycle or determines that the graph is not Hamiltonian.
\end{itemize}

Define the graph property $(P)$ as follows:
$$\forall U,W \subseteq V(G)\ disjoint\ subsets \,: e(U,W)> |U|\cdot |W| \cdot p\left( 1-\sqrt{\frac{n^{1.5}}{10|U|\cdot |W|}} \right) .$$
(In particular, if $|U|\cdot|W| \geq \frac{n^{1.5}}{10}$ then $e(U,W) \geq 1.)$\\

\begin{lemma}\label{prob2}
If  $p = p(n) \geq 70n^{-\frac{1}{2}}$ and $G\sim G(n,p)$, then $G \notin (P)$ is rare.
\end{lemma}

\begin{proof}
We bound from above the probability that $G \notin (P)$.\\
Let $U,W \subseteq V(G)$ be two disjoint sets, and assume that $|U|\cdot|W|\geq \frac{n^{1.5}}{10}$. By the Chernoff bound (Lemma \ref{chernoff}), the probability of $e(U,W) \leq |U|\cdot |W| \cdot p\left( 1-\sqrt{\frac{n^{1.5}}{10|U|\cdot |W|}} \right)$ is at most

		$$
		Pr\left[ Bin\left( |U|\cdot |W|,p\right) \leq |U|\cdot |W| \cdot p \left( 1-\sqrt{\frac{n^{1.5}}{10|U|\cdot |W|}} \right) \right] \leq \exp \left(- \frac{1}{20} \cdot n^{1.5}p \right) \leq e^{-3.5n}.\\
		$$
Finally, by the union bound we get that the probability that exist such $U,W$ is at most $3^n\cdot e^{-3.5n} = 2^{-2n} \cdot n^{-\omega (1)}$, as desired.
\end{proof}

\noindent In order to prove that \emph{CRE2} does not result in \emph{``failure"} on an input graph $G$ satisfying $(P)$ for $p=p(n)\geq 70n^{-\frac{1}{2}}$, we will show that none of the four stages that may result in \emph{``failure"} does so on such an input.\\
In the following lemmas it is assumed, without stating explicitly, that $p(n)\geq 70n^{-\frac{1}{2}}$.

\begin{lemma}\label{failure1}
Let $G$ be a graph on $n$ vertices satisfying $(P)$. Then \emph{Step 1} does not return \emph{``Failure"} on input $G$.
\end{lemma}

\begin{proof}
\emph{CRE2} fails this step if and only if $|\mathit{SMALL}(G)| \geq 2\sqrt{n}$. Let $A \subseteq \mathit{SMALL}(G)$ be some subset of size $2\sqrt{n}$. So $A$ and $V(G)\setminus A$ are two disjoint subsets with $|A|\cdot |V(G)\setminus A|\geq 1.9n^{1.5}$, but
$$
e\left( A,V(G)\setminus A \right) \leq 40\sqrt{n}|A| \leq \left( 1-\frac{1}{\sqrt{19}} \right) \cdot |A|\cdot |V(G)\setminus A|\cdot p,
$$
a contradiction to $G$ satisfying $(P)$.
\end{proof}

\begin{lemma}\label{failure4}
Let $G$ be a graph on $n$ vertices satisfying $(P)$. Then \emph{Step 4} does not return \emph{``Failure"} on input $G$.
\end{lemma}

\begin{proof}
 Say we failed to find a path of length at most $4$ between the vertices of some non-edge $e_i := (u_1,u_2)\in U\times U$ in the graph $H_i := G\setminus \left( \bigcup\limits _{j =1}^{i-1} P_j \cup \mathit{SMALL}(G) \cup U \right)$. Since $u_1, u_2\notin \mathit{SMALL}(G)$, it holds that
$$
|N_{H_i}(u_1)|,|N_{H_i}(u_2)| \geq 40\sqrt{n}-6\cdot |\mathit{SMALL}(G)| \geq 25\sqrt{n}.
$$
Let $D_2(G,v)$ denote the set of vertices in a graph $G$ of distance at most 2 from a vertex $v$. Because there is no path of length at most 4, the sets $D_2(H_i,u_1), D_2(H_i, u_2)$ do not intersect each other, which means that one of them, WLOG $D_2(H_i,u_1)$, is of size at most $\frac{1}{2}n$. But then we have
$$
|N_{H_i}(u_1)|\cdot |H_i\setminus \left( D_2(H_i,u_1) \cup \{u_1\} \right) | \geq 25\sqrt{n} \cdot \left( n-12\sqrt{n}-\frac{1}{2}n -1 \right) \geq 10n^{1.5},
$$
$$
e\left( N_{H_i}(u_1),H_i\setminus \left(N_{H_i}(u_1) \cup D_2(H_i,u_1) \cup \{u_1\}  \right) \right) =0,
$$
which means $G \notin (P)$, a contradiction.
\end{proof}

\begin{lemma}\label{failure5}
Let $G$ be a graph on $n$ vertices satisfying \emph{(P)}. Then \emph{Step 5} does not return \emph{``Failure"} on input $G$.
\end{lemma}

\begin{proof} Say we failed at some time $i$, that is: the constructed vertex set $U_i$ is an independent set. Recall that $U_i$ is the set of successors along $S_i$ of vertices in $N_G(V_i)\cap S_i$, where $V_i$ is a maximum sized connected component of $G\setminus S_i $. Let $W_i=N_G(V_i)\cap S_i$. Consider the following cases:

\begin{enumerate}

\item $|U_i| \geq n^{\frac{3}{4}}$. Let $A_1,A_2 \subseteq U_i$ be two disjoint subsets of size $\frac{1}{2}n^{\frac{3}{4}}$. So $|A_1|\cdot |A_2| = \frac{1}{4}n^{1.5}$, but $e(A_1,A_2)=0$, a contradiction.

\item $|V_i| > n - 30\sqrt{n}$. Observe two facts:
	\begin{itemize}
	\item By Def. \ref{small}, since $|V(G)\setminus V_i|< 30\sqrt{n} <40\sqrt{n}$, we get that $\forall v\in S_i\setminus \mathit{SMALL}(G):\ N_G(v)\cap V_i \neq \emptyset$;
	\item Since $|\mathit{SMALL}(G)| < \frac{1}{2}|S_0| \leq \frac{1}{2}|S_i|$, there are two vertices $w_1,w_2 \in S_i\setminus \mathit{SMALL}(G)$ such that $w_1=s_{S_i}(w_2)$.
	\end{itemize}
	So $w_1,w_2$ belong to $W_i$, and their successors are connected by an edge, which means that the algorithm could not have failed.

\item $15\sqrt{n} \leq |V_i| \leq n-30\sqrt{n}$. Observe that if the algorithm failed then $|W_i|=|U_i| \leq \min \lbrace \frac{1}{2}|S_i| , n^{3/4} \rbrace$, and therefore we have
	\begin{itemize}
	\item $|V_i| + |V(G)\setminus (V_i \cup W_i)| \geq n-n^{\frac{3}{4}}$;
	\item $|V_i| \geq  15\sqrt{n}$;
	\item $|V(G)\setminus (V_i \cup W_i)| = |V(G)\setminus V_i | - | W_i| \geq |V(G)\setminus V_i | - \frac{1}{2} | S_i| \geq \frac{1}{2} |V(G) \setminus V_i| \geq 15\sqrt{n}$.
	\end{itemize}
	So $V_i$ and $V(G) \setminus (V_i \cup W_i)$ are two sets, with $|V_i|\cdot |V(G) \setminus (V_i \cup W_i)| \geq 10n^{1.5}$, but $e(V_i,V(G)\setminus (V_i \cup W_i))=0$, a contradiction to our assumption that $G\in (P)$.

\item $|V_i|\leq 15\sqrt{n},\ |S_i| < 0.99n$. Then all connected components of $G\setminus S_i$ are of size at most $15\sqrt{n}$, and the sum of their sizes is at least $0.01n$. So the vertices of $V(G)\setminus S_i$ can be partitioned into two sets $A_1,A_2$ such that each one of them is a union of connected components, and $|A_1|,|A_2| \geq n^{\frac{3}{4}}$. But then $|A_1|\cdot |A_2| \geq n^{1.5}$ and $e(A_1,A_2)=0$, a contradiction.

\end{enumerate}

\noindent The complementing case to those already covered is when $|V_i|\leq 15\sqrt{n},\ |S_i| \geq 0.99n$, which can only occur in Stage 6.
\end{proof}

\begin{lemma}\label{failure6}
Let $G$ be a graph on $n$ vertices satisfying \emph{(P)}. Then \emph{Step 6} does not return \emph{``Failure"} on input $G$.
\end{lemma}

\begin{proof}

We show that under the assumption that $G\in (P)$, the cycle $S_i$ contains two vertices $u,w$ and two edge subsets $E_1,E_2$ as described in \emph{Step 6}. Since the algorithm searches for such $u,w,E_1,E_2$ exhaustively, and only returns \emph{``Failure"} upon failing the search, this means that if $G\in (P)$ the algorithm does not fail. \\
Recall that in this stage we can assume that $|S_i| \geq 0.99n$ and that all connected components of $G\setminus S_i$ are of size at most $15\sqrt{n}$. It follows that for every $v\in V(G)\setminus S_i$ we have $|N_G(v)\cap S_i| \geq d_G(v)-|V_i| \geq 20\sqrt{n}$. Observe that if $ |N_G(v)\cap S_i| > \frac{1}{2}n$ then $v$ has two neighbours adjacent on $S_i$, say $u,w$, so setting $E_1 = (u,w),\ E_2=\emptyset$ results in a cycle as desired, so we can assume that $|N_G(v)\cap S_i| \leq \frac{1}{2}n$.\\
Let $U_i := p(N_G(v)\cap S_i)$. Since $|U_i|,|S_i\setminus U_i|\geq 20\sqrt{n}$ and $|U_i|+|S_i\setminus U_i|\geq 0.99n$, we get that $|U_i|\cdot |S_i\setminus U_i| \geq 10n^{1.5}$, and therefore $e(U_i,S_i\setminus U_i) \geq 0.9p|U_i|\cdot |S_i\setminus U_i| \geq 0.4|U_i|np$. It follows that there is some $u \in N_G(v)\cap S_i$ such that $d_{S_i}(p(u)) \geq 0.4np \geq 20\sqrt{n}$. Denote $t_0:=p(u)$, and let $Q$ be the path $\{v\} \cup S_i(u\rightarrow t_0)$.\\
Define the following three special vertices on $Q$:
\begin{itemize}
\item $c_v$ : a vertex on $Q$ such that $|N_{Q(v\rightarrow c_v)}(v)| = \lfloor \frac{1}{2}|N_{Q}(v)| \rfloor$;
\item $c_{t_0}$ : a vertex on $Q$ such that $|N_{Q(c_{t_0}\rightarrow t_0)}(t_0)| = \lfloor \frac{1}{2}|N_{Q}(t_0)| \rfloor$;
\item $c$ : a vertex on $Q$ such that $|Q(v\rightarrow c)| = \lfloor \frac{1}{2}|Q| \rfloor$.
\end{itemize}
We will assume that $c_v$ and $c$ precede $c_{t_0}$ on $Q$, and remark that the proof is quite similar for the complementing cases, in which $c_{t_0}$ precedes one or both of $c_v,c$, with some minor changes required to some of the definitions down the line.\\
Denote: $Q_1 := Q(v\rightarrow c_v),\ Q_2:= Q(c_{t_0}\rightarrow t_0),\ Q_3:= Q(v\rightarrow c)$.\\
We now aim to show that $E_1,E_2,w$ as required exist in the graph, with respect to the already chosen $u$, by using rotations and extensions.\\
Let $W_i$ be the set $N_{Q_1}(v)$ and $T_i$ the set $s_Q(N_{Q_2}(t_0))$. By our choices of $v,t_0,c_v,c_{t_0}$ we know that $|W_i|,|T_i| \geq 10\sqrt{n}$. Now, construct the set $O_i$ as follows:\\
For each vertex $x\in W_i$ and for each $y\in N_{Q_3}(p_Q(x))\setminus \{x\}$ add $s_Q(y)$ to $O_i$ if $y \in Q(v\rightarrow x)$ and add $p_Q(y)$ to $O_i$ if $y \in Q(x\rightarrow c)$.
\begin{clm}
The size $|O_i|$ is at least $0.2n$.
\end{clm}
\begin{proof}
By our construction, $|O_i| \geq |N_{Q_3}(p_Q(W_i))| - |W_i|$. If $|O_i| < 0.2 n$ then $| Q_3 \setminus N_{Q_3}(p_Q(W_i))| \geq 0.25 n - |W_i|$, and $p_Q(W_i),\ Q_3 \setminus N_{Q_3}(p_Q(W_i))$ are two sets that have no edges between them, but the product of their sizes is at least $2n^{1.5} $, a contradiction.
\end{proof}

\begin{clm}
There is an edge between $O_i$ and $T_i$.
\end{clm}
\begin{proof}
The two sets are disjoint, and $|O_i|\cdot |T_i| \geq 2n^{1.5}$.
\end{proof}

\begin{figure}[h]
    \begin{tikzpicture}[vertex/.style={draw,circle,color=black,fill=black,inner sep=1}]
      
      \pgfmathsetmacro{\alpha}{15}
      \pgfmathsetmacro{\r}{3}
      \pgfmathsetmacro{\rth}{0.707}
      \pgfmathsetmacro{\rttq}{0.866}
      \pgfmathsetmacro{\sinsf}{0.966}
      \pgfmathsetmacro{\cossf}{0.259}
      
      \node[vertex,label=above:$v$] (v) at (-0.2*\r, 1.5*\r) {};
      \node[vertex,label=below:$u$] (u) at (-0.5*\r, \rttq*\r) {};
      \node[vertex,label=above:$p_Q(w)$] (pw) at (0, \r) {};
      \node[vertex,label=below:$w$] (w) at (\cossf*\r, \sinsf*\r) {};
      \node[vertex,label=below:$c_v$] (cv) at (\rth*\r, \rth*\r) {};
      \node[vertex,label=right:$s$] (s) at (\r, 0) {};
      \node[vertex,label=right:$s_Q(s)$] (ss) at (\sinsf*\r, -\cossf*\r) {};
      \node[vertex,label=right:$c$] (c) at (\rth*\r, -\rth*\r) {};
      \node[vertex,label=below:$c_{t_0}$] (ct0) at (0, -\r) {};
      \node[vertex,label=left:$p_Q(t)$] (pt) at (-\sinsf*\r, -\cossf*\r) {};
      \node[vertex,label=left:$t$] (t) at (-\r, 0) {};
      \node[vertex,label=left:$p_{S_i}(u)\textrm{=}t_0$] (t0) at (-\rth*\r, \rth*\r) {};

      \draw[thick] (\r,0) arc [start angle=0, end angle=90-\alpha, radius=\r];
      \draw[thick] (0,\r) arc [start angle=90, end angle=135-\alpha, radius=\r];
      \draw[thick,dotted] (0,\r) arc [start angle=90, end angle=90-\alpha, radius=\r];
      \draw[thick,dotted] (\r,0) arc [start angle=0, end angle=-\alpha, radius=\r];
      \draw[thick,dotted] (-\rth*\r,\rth*\r) arc [start angle=135, end angle=135-\alpha, radius=\r];
      \draw[thick] (-\r,0) arc [start angle=180, end angle=135, radius=\r];
      \draw[thick,dotted] (-\r,0) arc [start angle=180, end angle=180+\alpha, radius=\r];
      \draw[thick] (0,-\r) arc [start angle=270, end angle=180+\alpha, radius=\r];
      \draw[thick] (0,-\r) arc [start angle=-90, end angle=-\alpha, radius=\r];
      
	  \draw[bend left=30,thick,densely dashed] (ss) to (pw);
      \draw[bend left=10,thick, densely dashed] (t) to (s);
      \draw[bend right=30,thick, densely dashed] (pt) to (t0);
      \draw[bend left=20,thick] (v) to (w); 
      \draw[bend right=20,thick] (v) to (u);

    \end{tikzpicture}
  \caption{Extention of cycle $S_i$ (oriented clockwise) to cycle $S_{i+1}$ that includes $v$, by removing the edges of $E_1$ (dotted) and adding the edges of $E_2$ (dashed) and $(v,u),(v,w)$.}\label{fig1}
\end{figure}

Let $s\in O_i,\ t\in T_i$ be such that $(s,t)\in E(G)$, and let $w \in W_i$ be a vertex that caused $s$ to be added to $O_i$. Finally, define:
\begin{itemize}
\item $E_1 := \{(u,t_0),\ (p_Q(w),w),\ (s,s_Q(s)),\ (p_Q(t),t) \}$;
\item $E_2 := \{(p_Q(w),s_Q(s)),\ (p_Q(t)),t_0),\ (s,t)\}$.
\end{itemize}

Then $E_1,E_2,u,w$ are as required by the algorithm (see Fig. 1 for illustration).

\end{proof}

\section{Concluding remarks} \label{remarks}
To summarise, we have presented an algorithm \emph{CRE} which is comprised of three aligned algorithms, in the spirit of previous results, and utilises rotations and extensions in order to find a Hamilton cycle in a graph, and proved that its expected running time on a random graph $G\sim G(n,p)$ is optimal, for $p\geq 70n^{-\frac{1}{2}}$.

We note that even if we make changes to some parameters in our algorithm, $p=\Omega \left( n^{-\frac{1}{2}} \right)$ seems to be the lowest range of probability for which our expected running time bound works, at least with our current argument. The reason for this is the existence of some bottlenecks along the proof, where smaller orders of magnitude of the edge probability no longer work. Such a bottleneck can be observed, for example, in Step 4 of \emph{CRE2}, where the algorithm tries to connect some set of paths into a cycle that contains them, by finding paths between pairs of endpoints of paths one by one. In our proof we use the fact that the total length of the paths is highly likely to be much smaller than the minimum degree of the vertices at the endpoints of the paths (that is to say that the complement event is rare, i.e., has probability $2^{2n}\cdot n^{-\omega (1)}$). This is due to the fact that, on the one hand, all of the paths' endpoints have degrees at least comparable to the expected average degree of the graph, since by our construction none of the endpoints are members of $\mathit{SMALL}(G)$ -- the set of vertices with very small degrees. On the other hand, the total number of vertices in the union of all the paths is not likely to be very big, since this vertex set contains at most $6\cdot |\mathit{SMALL}(G)|$ vertices, a size likely to be much smaller than the average degree of the graph for our parameters, as we observed that $\mathit{SMALL}(G)$ is highly likely to be of size much smaller than $np$. If $p =o\left( n^{-\frac{1}{2}} \right)$, however, then the event ``$|\mathit{SMALL}(G)|>np$" has probability $2^{-o(n)}$, and in particular it is no longer rare. In other words, the probability that one of the paths' endpoints has all its neighbours residing in the union of $\mathit{SMALL}(G)$ and previously constructed paths is $2^{-o(n)}$, and the expected runtime of \emph{CRE} might no longer even be polynomial.\\
And so, we leave it as an open question whether a polynomial expected running time Hamiltonicity algorithm exists for edge probability $p =o\left( n^{-\frac{1}{2}} \right)$.\\

\noindent \textbf{Acknowledgements.} The authors would like to express their thanks to the referees of the paper, and to Samotij Wojtek, for their valuable input towards improving the presentation of our result.\\

\end{document}